\documentclass[11pt,letterpaper]{amsart}

\usepackage[margin=3cm]{geometry}
\usepackage{amssymb}
\usepackage{mathtools}

\usepackage{enumitem}
\setlist[enumerate,1]{label={\normalfont(\roman*)}}
\usepackage{mathrsfs}
\usepackage[only,llbracket,rrbracket]{stmaryrd}
\usepackage[svgnames]{xcolor}
\usepackage{verbatim}

\usepackage{tikz}

\usepackage[bookmarks=true]{hyperref}
\hypersetup{colorlinks=true}  

\usepackage{cleveref}

\usepackage{soul}

\definecolor{burntorange}{rgb}{0.8, 0.33, 0.0}

\theoremstyle{plain}
\newtheorem{thm}{Theorem}[section]
\newtheorem{lem}[thm]{Lemma}
\newtheorem{prop}[thm]{Proposition}
\newtheorem{cor}[thm]{Corollary}
\newtheorem{thmABC}{Theorem}

\theoremstyle{definition}

\theoremstyle{remark}

\newcommand{\noshow}[1]{}

\newcommand{\Aut}[1]{\operatorname{Aut}(#1)}
\newcommand{\Core}[2]{\operatorname{Core}_{#1}(#2)}

\newcommand{\Sym}[1]{\operatorname{Sym}(#1)}

\newcommand{\GL}[2][]{\mathrm{GL}_{#1}(#2)}

\newcommand{\cU}{\mathcal{U}}

\renewcommand{\leq}{\leqslant}  
\renewcommand{\geq}{\geqslant}

\newcommand{\emptyword}{\varepsilon}

\newcommand{\Tree}{\mathcal{T}}
\newcommand{\topo}[1]{\tau_{\text{#1}}}

\newcommand{\0}{\mathbf{0}}
\newcommand{\1}{\mathbf{1}}
\newcommand{\ad}{\operatorname{ad}}
\newcommand{\Field}[1]{\mathbb{F}_{#1}}
\newcommand{\half}{\tfrac{1}{2}}
\newcommand{\invlim}{\varprojlim}
\newcommand{\Laurentseries}[2]{#1(\!(#2)\!)}
\newcommand{\localring}[1]{\mathcal{O}_{#1}}
\newcommand{\Nat}{\mathbb{N}}
\newcommand{\padics}[1][p]{\mathbb{Q}_{#1}}
\newcommand{\padicZ}[1][p]{\mathbb{Z}_{#1}}

\newcommand{\set}[2]{\{\,#1\mid#2\,\}}
\newcommand{\Span}[2][]{\operatorname{Span}_{#1}(#2)}

\newcommand{\units}[1]{#1^{\ast}}
\newcommand{\Zint}{\mathbb{Z}}

\newcommand{\nbd}{\nobreakdash-}

\usepackage[style=alphabetic,sorting=nyt,maxbibnames=7,maxcitenames=5,block=space,backend=bibtex,doi=false,isbn=false,url=false]{biblatex}
\renewbibmacro{in:}{%
  \ifentrytype{article}{}{\printtext{\bibstring{in}\intitlepunct}}}

\AtEveryBibitem{
  \clearlist{language}
}

\title[Uncountably many local isomorphism types of simple groups]{
Uncountably many local isomorphism types of compactly generated simple groups}

\date{\today}

\author[]{Ilaria Castellano}
\address{Ilaria Castellano: Dipartimento di Matematica e Applicazioni, Università degli Studi di Milano-Bicocca, Via Roberto Cozzi, 55 - 20126, Milano, Italy}
\email{ilaria.castellano@unimib.it}

\author[]{Jorge Fariña-Asategui}
\address{Jorge Fariña-Asategui: Section de math\'{e}matiques, Universit\'{e} de Gen\`{e}ve, Rue du Conseil-G\'{e}n\'{e}ral 7--9, 1205 Geneva, Switzerland}
\email{Jorge.Farina-Asategui@unige.ch}

\author[]{Mikel Eguzki Garciarena}
\address{Mikel Eguzki Garciarena: Heinrich-Heine-Universität Düsseldorf, Mathematisch-Naturwissenschaftliche Fakultät, Mathematisches Institut, Universitätsstr. 1, 40225 Düsseldorf, Germany.}
\email{mikel.eguzki.garciarena.perez@hhu.de}

\author[]{Bianca Marchionna}
\address{Bianca Marchionna: Université catholique de Louvain (IRMP), Chemin du Cyclotron 2, 1348 Louvain-la-Neuve, Belgium}
\email{bianca.marchionna@uclouvain.be}

\author[]{Martyn Quick}
\address{Martyn Quick: Mathematical Institute, University of St Andrews, North Haugh, St Andrews, Fife, KY16 9SS, U.K.}
\email{mq3@st-andrews.ac.uk}

\author[]{Colin D.~Reid}
\address{Colin D.~Reid: School of Computer, Data and Mathematical Sciences, Western Sydney University, Penrith NSW 2751, Australia}
\email{C.Reid5@westernsydney.edu.au}

\author[]{Simon M.~Smith}
\address{Simon M.~Smith: Charlotte Scott Research Centre for Algebra, University of Lincoln, Lincoln, U.K.}
\email{sismith@lincoln.ac.uk}

\author[]{Stephan Tornier}
\address{Stephan Tornier: The University of Newcastle, School of Computer and Information Sciences, University Drive, 2308 Callaghan NSW, Australia}
\email{stephan.tornier@newcastle.edu.au}

\author[]{Matteo Vannacci}
\address{Matteo Vannacci: Dipartimento di Matematica `Ulisse Dini', Universit\`a degli Studi di Firenze, Viale Morgagni 67/A, 50134 Firenze, Italy}
\email{matteo.vannacci@unifi.it}

\author[]{John S.~Wilson}
\address{John S.~Wilson: Christ's College, Cambridge, CB2 3BU, Great Britain
\and 
Mathematisches Institut, Universit\"at Leipzig, 04109 Leipzig, Deutschland}
\email{jsw13@cam.ac.uk, john.wilson@uni-leipzig.de}

\begin{document}

\begin{abstract}
A major open question in the theory of locally compact groups is the following. Let $\mathscr{S}$ be the class of non-discrete compactly generated totally disconnected locally compact groups that are topologically simple. Is the number of local isomorphism classes of groups in $\mathscr{S}$ uncountable? We have answered this question, showing that there are $2^{\aleph_0}$ local isomorphism classes in $\mathscr{S}$.

Our result was obtained without the use of artificial intelligence; it arose from a problem session that ran over several days at the workshop {\it Branch groups: subgroups, rigidity, topologies} at the Universidad Complutense de Madrid, organised by Dominik Francoeur, Alejandra Garrido and Tatiana Nagnibeda.
\end{abstract}
\maketitle

\section{Introduction}

The theory of locally compact groups describes the symmetries of all locally finite structures. It typically begins with the observation that every locally compact group is an extension of a connected locally compact group by a totally disconnected locally compact group (henceforth, \emph{tdlc} group). The connected locally compact groups are well-understood: by the solution to Hilbert's fifth problem by Gleason, Montgomery–Zippin and
Yamabe, such groups are known to be projective limits of Lie groups, and the application of Lie theoretic techniques to all connected locally compact groups gave rise to deep decomposition and classification theorems. Consequently, the focus of research in locally compact groups now primarily concerns understanding tdlc groups.

In general tdlc groups cannot be approximated by Lie groups, but the lack of Lie theoretic tools is compensated for by a rich structure of compact open subgroups: a classical theorem of D.~van Dantzig~\cite{vDa31} guarantees that every tdlc group admits a family of compact open subgroups (i.e.,~profinite subgroups) that form a basis of identity neighbourhoods. This is in stark contrast to the connected case: connected locally compact groups have no proper open subgroups. This rich \emph{local} subgroup structure of tdlc groups forms a foundation for the majority of the tools used to analyse tdlc groups, including Willis Theory. It also allows techniques from permutation group theory to be applied to tdlc groups via the following correspondence: every tdlc group is isomorphic as a topological group to a closed permutation group in which every point stabilizer has finite orbits, and conversely every such permutation group admits a natural topology under which it is a tdlc group.

The local subgroup structure of tdlc groups confers significant potency to the local isomorphism relation, where two tdlc groups $G_1, G_2$ are \emph{locally isomorphic} if they have topologically isomorphic open subgroups $U_1 \leq G_1$ and $U_2 \leq G_2$. Properties common to all sufficiently small compact open subgroups are called \emph{local}, and a prominent theme in tdlc theory involves understanding how the local and global properties of tdlc groups are related. For example, the structure lattice of a tdlc group $G$ (see \cite{CRW-LocNormal2}) is known to depend only on the local isomorphism class of $G$.

The most mature area of research in tdlc theory concerns compactly generated tdlc groups. Every locally compact group is the directed union of its compactly generated subgroups, so questions about local properties can be reduced to the compactly generated case. In 2011, Pierre-Emmanuel Caprace and Nicolas Monod~\cite{CapraceMonod} showed that the structure of non-discrete compactly generated tdlc groups depends on the class $\mathscr{S}$ of non-discrete, compactly generated, topologically simple, tdlc groups. In \cite{ReidWesolek:Essentially}, Phillip Wesolek and Colin Reid go further and prove that every compactly generated locally compact group admits a chief series of finite length up to compact groups and discrete groups; such a series is called an \emph{essentially chief series}. A Jordan--H\"{o}lder theorem holds for these essentially chief series. These results place the class $\mathscr{S}$ at the heart of tdlc theory, a position emphasised by Pierre-Emmanuel Caprace at the 7th European Congress of Mathematics~\cite{CapraceSimple}.

In 2017, Simon Smith~\cite{smith17} proved that there are $2^{\aleph_0}$ many isomorphism classes in $\mathscr{S}$, using a construction he called the \emph{box product}. The construction takes as inputs two permutation groups, and produces another permutation group that inherits many of the permutational and topological properties of the input groups, but it does not inherit discreteness. Results in \cite{smith17} suggest that the problem of classifying groups in $\mathscr{S}$ up to isomorphism is intractable. In light of this, the following became one of the main open questions in tdlc theory:

\begin{quote}
{\it Is the number of local isomorphism classes of groups in $\mathscr{S}$ uncountable?}
\end{quote}

The question appears as \S20.5.1 in the lecture notes \cite{NewDirectionsBook} which were produced following the 2014 Oberwolfach Arbeitsgemeinschaft {\it Totally disconnected groups}; it also appears in the Kourovka Notebook \cite{KourovkaNotebook} as Problem 19.22 where it is attributed to Pierre-Emmanuel Caprace; and it is referred to in George Willis and Pierre-Emmanuel Caprace's survey article \cite[pp.~1563]{CapraceWillis:ICM} that accompanies their 2022 ICM talk. We affirmatively answer this question, proving the following.

\begin{thmABC} 
\label{thm:uncount_simple}
There exist $2^{\aleph_0}$ many non-discrete compactly generated  simple tdlc groups which are not pairwise locally isomorphic.   \end{thmABC}

In proving Theorem~\ref{thm:uncount_simple}, we shed some light on the possible pro-$p$ groups and linear groups over local fields that can appear as closed subgroups of groups in $\mathscr{S}$.

\section*{Acknowledgements}
This work arose from a problem session that took place over several days at the workshop {\it Branch groups: subgroups, rigidity, topologies} at the Universidad Complutense de Madrid. The workshop was organised by Dominik Francoeur, Alejandra Garrido and Tatiana Nagnibeda and it was supported by the Spanish Ministry of Science (grant code PID2024-155800NB-C31) and by the European Research Council (grant code SATURN, 101076148). The authors would like to express their gratitude and sincere thanks to the organisers for creating an environment where this work could flourish.

\section*{AI Statement}

No artificial intelligence was used to prove our result.

\section{Preliminaries}\label{sec:prelim}

\subsection{Topological groups}\label{sec:topological_groups}

There will be two sources for topological groups in this paper: profinite groups and permutation groups under the \emph{permutation topology}. For an introduction to the former see \cite{wilson_profinite} and for the latter see \cite{Mol02}.

A permutation group (which for us will always act from the right) is a group $G$ together with a \emph{faithful} action on some set $\Omega$; that is, an action in which the pointwise stabilizer of $\Omega$ in $G$ is trivial. The group consisting of all permutations of $\Omega$ is $\Sym{\Omega}$ and we consider $G$ to be a subgroup of $\Sym{\Omega}$.
For any subset $\Delta \subseteq \Omega$ we denote the setwise stabilizer and pointwise stabilizer of $\Delta$ in $G$ by $G_{\{\Delta\}}$ and $G_{(\Delta)}$ respectively.

If $\Omega$ is any non-empty set, then $\mathrm{Sym}(\Omega)$ can be given a natural topology called the \emph{permutation topology}. This is the topology of pointwise convergence, under which $\Sym{\Omega}$ is a Hausdorff topological group. A neighbourhood basis of the identity is given by pointwise stabilizers of finite subsets of $\Omega$. Under this topology, if $G \leq \Sym{\Omega}$ then a subgroup of $G$ is open in $G$ if and only if it contains the pointwise stabilizer in $G$ of some finite subset of~$\Omega$.
For any finite subset $\Delta\subset \Omega$, the group $\mathrm{Sym}(\Omega)_{(\Delta)}$ is both open and closed in $\mathrm{Sym}(\Omega)$, so with this topology $\mathrm{Sym}(\Omega)$ is totally disconnected.

Recall that $G \leq \Sym{\Omega}$ is closed if and only if some point stabilizer is closed, and this holds if and only if all point stabilizers in $G$ are closed. Our group $G$ is compact if and only if $G$ is closed and all $G$-orbits on $\Omega$ are finite. If $G$ is closed then it is locally compact if and only if $G_{(\Delta)}$ is compact for some finite $\Delta \subseteq \Omega$.

If $G$ is now any group (not necessarily a permutation group) and $H \leq G$ then there is a natural transitive action of $G$ on the right coset space $\Omega := H \backslash G$ via right multiplication. The subgroup $H$ is a point stabilizer in this action and the pointwise stabilizer of $\Omega$ in $G$ (i.e.,~the kernel of the action) is the \emph{core} of $H$ in $G$, defined as $\Core{G}{H} \coloneq \bigcap_{g \in G} g^{-1}Hg$. Thus, if $H$ is \emph{core-free} in $G$ (by which we mean $\Core{G}{H}$ is trivial) then $G$ acts faithfully on $H \backslash G$ and we identify $G$ with its image in $\Sym{H \backslash G}$, viewing $G$ as a permutation group.

If $G$ is a tdlc topological group and $H$ is a compact and open subgroup then the above coset action of $G$ on $H \backslash G$ has a kernel that is closed and compact in $G$. The subgroup $\hat{G}$ of $\Sym{H \backslash G}$ induced by this action is called the \emph{Schlichting completion} (see \cite{ReidWesolekCompletions}) of the pair $(G,H)$ and is closed and tdlc under the permutation topology on $\Sym{H \backslash G}$. Moreover, by the Orbit-Stabilizer Theorem, all orbits of stabilizers in $\hat{G}$ have cardinality $|g^{-1}Hg : (g^{-1}Hg) \cap (k^{-1}Hk)|$ for some $g,k \in G$, and this is always finite because $(g^{-1}Hg) \cap (k^{-1}Hk)$ is open and $g^{-1}Hg$ is compact. All stabilizers in $\hat{G}$ have only finite orbits and are closed because $\hat{G}$ is closed; thus they are compact in the permutation topology.
A particular instance of this phenomenon that is important for our article is when $H$ is core-free. In this situation the action is faithful and the identification of  $G$ with $\hat{G} \leq \Sym{H \backslash G}$ makes sense because the two groups are topologically isomorphic; we can thus consider $G$ to be a permutation group that is closed, tdlc, with compact open stabilizers (under the permutation topology). As a permutation group $G$ is \emph{subdegree-finite}; that is, all orbits of stabilizers in $G$ are finite.

Recall that a profinite group is a Hausdorff, compact and totally disconnected group and two profinite groups are  \emph{commensurable} if they have isomorphic open subgroups. If $G$ is tdlc and $H$ is compact open and core-free then (viewing $G$ as a permutation group) all point stabilizers are profinite in the permutation topology because $H$ is profinite, and they are pairwise commensurable. The pairwise commensurability can be seen directly from their behaviour as permutation groups: subdegree-finiteness implies that the intersection of two stabilizers has finite index in each stabilizer, with the index equal to their respective orbit lengths via the Orbit-Stabilizer Theorem.

\subsection{Wreath products}
\label{section:WreathProds}

Let $G \leq \Sym{X}$~and~$H \leq \Sym{Y}$ be permutation groups on sets $X$~and~$Y$.  We regard them as topological groups with respect to the permutation topology.  Recall that the \emph{wreath product} $J = G \wr_{Y} H$ is constructed as a semidirect product $B \rtimes H$ where the \emph{base group} $B = G^{Y}$ is a direct product of copies of~$G$ indexed by the set~$Y$ and $H$~acts on~$B$ by permuting the factors in the same way that $H$~permutes the points of~$Y$; that is, if $b = (g_{y})$ is an element of~$B$ (written as a sequence indexed by the set~$Y$) and $h \in H$ then
\[
h^{-1}bh = (g_{y^{h^{-1}}})_{y \in Y}.
\]
We shall write elements of~$J$ as products~$bh$ where $b \in B$ and $h \in H$.  There is a faithful action of~$J$ on the Cartesian product $X \times Y$ given by
\begin{equation}
(x,y)^{bh} = (x^{g_{y}},y^{h}),
\label{eq:WreathAction}
\end{equation}
where $b = (g_{y})_{y \in Y}$ with $g_{y} \in G$ for each~$y$ and $h \in H$.  Thus $J$~can be viewed as a subgroup of~$\Sym{X \times Y}$.  Observe then that the stabilizer of the point~$(x,y)$ is given by
\[
J_{(x,y)} = \biggl( G_{x} \times \prod_{z \neq y} G \biggr) H_{y}
\]
(where in this notation $G_{x}$~occurs in the $y$th coordinate). Thus
 every point stabilizer is open in~$J$ if we equip it with the product topology induced from the topologies on~$G$ and on~$H$.

Conversely, observe that the collection~$\mathscr{U}$ of all sets of the form~$BH_{y}$, for $y \in Y$, together with those of the form $\left( G_{x} \times \prod_{z \neq y} G \right)H$, for $x \in X$, has the property that $\set{ Ug }{ U \in \mathscr{U}, \, g \in J }$ is a subbasis for the product topology on~$J$.  Each~$U$ in~$\mathscr{U}$ contains some point stabilizer~$J_{(x,y)}$ and hence every subset of~$J$ that is open with respect to the product topology is also open with respect to the permutation topology induced from $\Sym{X \times Y}$.  We conclude that the product topology on the wreath product $J = G \wr_{Y} H$ coincides with the permutation topology induced upon it.

Now consider a sequence~$(G_{i})_{i \in \Nat}$ of permutation groups each acting, respectively, on some set~$X_{i}$.  Again we consider them as topological groups with respect to the permutation topology.  Define $J_{1} = G_{1}$ viewed as a permutation group on $Y_{1} = X_{1}$.  Suppose, as an inductive hypothesis, that for some $i > 1$ we have constructed a wreath product~$J_{i-1}$ acting as a permutation group on $Y_{i-1} = X_{i-1} \times X_{i-2} \times \dots \times X_{1}$.  Define $J_{i} = G_{i} \wr_{Y_{i-1}} J_{i-1}$ viewed, as above, as a permutation group on $Y_{i} = X_{i} \times Y_{i-1}$.  There is a natural projection map $\pi_{i} \colon J_{i} \to J_{i-1}$ determined by the semidirect product construction and hence we construct an inverse limit of wreath products
\[
\dots \to J_{3} \to J_{2} \to J_{1} \to \1.
\]
Define $W = \invlim J_{i}$ to be the inverse limit of this system.  We shall call~$W$ an \emph{iterated wreath product} and denote it by
\[
W = \dots \wr_{Y_{3}} G_{3} \wr_{Y_{2}} G_{2} \wr_{Y_{1}} G_{1}.
\]
As is common, we identify~$W$ here with the subgroup
\begin{equation}
W = \biggl\{ \, (g_{i}) \in \prod_{i=1}^{\infty} J_{i} \,\biggm|\, \text{$g_{i}\pi_{i} = g_{i-1}$ for $i \geq 2$} \, \biggr\}
\label{eq:invlimit}
\end{equation}
of the Cartesian product of the topological groups~$J_{i}$.  Note here that the $i$th coordinate of $(g_i) \in W$ is $g_i \in J_i$. This identification then endows the inverse limit with the subspace topology induced by the product topology on $\prod_{i=1}^{\infty} J_{i}$.  (If each~$J_{i}$ were finite then this, of course, endows~$W$ with the structure of a profinite group.  More generally, if each~$J_{i}$ is a profinite group then the iterated wreath product is itself profinite.)

We may construct a rooted tree~$\Tree$ whose vertices are the points in the sets~$Y_{i}$ together with a root~$\emptyword$.  The root~$\emptyword$ is joined to every vertex in~$Y_{1}$, while if $i \geq 2$ and $z = (x,y) \in Y_{i}$ where $x \in X_{i}$ and $y \in Y_{i-1}$ then there is an edge from~$y$ to~$z$.  There is a natural action of the iterated wreath product~$W$ on the rooted tree~$\Tree$ where an element $g = (g_{i}) \in W$ (as in Equation~\eqref{eq:invlimit}) moves a vertex~$y \in Y_{j}$ in the $j$th level of the tree by applying its $j$th coordinate~$g_{j}$.  The compatibility condition on the coordinates of~$g$ together with the definition~\eqref{eq:WreathAction} of the action of the wreath product on pairs ensures that this is indeed an action of~$W$ on~$\Tree$ by automorphisms.  Furthermore, since each~$J_{i}$ acts faithfully on the set~$Y_i$, it follows that the iterated wreath product~$W$ acts faithfully on the set~$V\Tree$ of vertices of~$\Tree$ and we may therefore (initially as an abstract group) view~$W$ as a subgroup of the automorphism group of $\Tree$; that is, $W \leq \Aut{\Tree}$. There are now two topologies on $W$: the aforementioned induced topology~$\topo{ind}$ and the permutation topology~$\topo{perm}$ arising from $\Aut{\Tree}$.

If $y \in Y_{j}$, then the stabilizer~$W_{y}$ in~$W$ of the vertex~$y$ consists of those elements~$(g_{i})$ of~$W$ such that $g_{j}$~lies in the stabilizer of~$y$ in~$J_{j}$.  Thus
\begin{equation}
W_{y} = W \cap ( J_{1} \times J_{2} \times \dots \times J_{j-1} \times (J_{j})_{y} \times J_{j+1} \times \cdots )
\label{eq:WreathStabilizer}
\end{equation}
is $\topo{ind}$\nbd open in~$W$.  Conversely, since each~$J_{i}$ is equipped with the permutation topology, every $\topo{ind}$\nbd open neighbourhood of the identity in~$W$ contains an intersection of finitely many subgroups of the form appearing on the right-hand side of Equation~\eqref{eq:WreathStabilizer}.  Therefore every $\topo{ind}$\nbd open subset of~$W$ is also $\topo{perm}$\nbd open.  In conclusion, the natural topology~$\topo{ind}$ on the iterated wreath product~$W$ coincides with the permutation topology~$\topo{perm}$ via this faithful action on~$\Tree$, and we may thus consider $W \leq \Aut{\Tree}$ as a topological group.

\subsection{Groups acting on bi-regular trees and their stabilizers}\label{sec:groups_acting_on_trees}

Here we briefly recall some relevant aspects of the box product construction and universal locally-$(M, N)$ groups first introduced in \cite{smith17}. We warn the reader that our colour maps and permutations are applied from the right, but in \cite{smith17} they are applied from the left.

Let $X$ and $Y$ be sets of cardinality at least $2$, and let $M\leq\Sym{X}$ and $N\leq\Sym{Y}$ be permutation groups, with $M$ or $N$ non-trivial.
Let $\Tree$ be the $(|X|,|Y|)$-biregular tree. Of course this means that the tree $\Tree$ has no root. Edges in $\Tree$ are undirected and consist of two arcs, one in each direction. The edge set of $\Tree$ is $E\Tree$, the vertex set is $V\Tree$ and the arc set is $A\Tree$. For $v \in V\Tree$ the set of arcs originating from $v$ is $A(v)$ and the set of arcs terminating at $v$ is $\overline{A}(v)$. We bipartition the vertex set $V\Tree$ as $\Omega_X\sqcup \Omega_Y$, where the vertices in $\Omega_X$ and in $\Omega_Y$ have valency $|X|$ and $|Y|$ respectively. A subgroup $G \leq \Aut{\Tree}$ is said to be \emph{locally-$(M, N)$} if it preserves setwise the two parts $\Omega_X$ and $\Omega_Y$ of the natural bipartition of $\Tree$ and moreover for all $v \in \Omega_X$ (resp.~$v \in \Omega_Y$) the stabilizer $G_v$ induces $M$ (resp.~$N$) on the distance one neighbours of $v$ in $\Tree$.

A \emph{legal colouring} is a map $\mathcal{L}\colon AT\rightarrow X\sqcup Y$ such that for all $v \in \Omega_X$ (resp.~$v \in \Omega_Y$) the restriction of $\mathcal{L}$ to $A(v)$, which we denote here by $\lambda_{v}^{\mathcal L}: A(v) \rightarrow X$ (resp.~$\lambda_{v}^{\mathcal L}: A(v) \rightarrow Y$), is a bijection, and the restriction of $\mathcal{L}$ to $\overline{A}(v)$ is constant.

For any automorphism preserving the parts of the bipartition $g \in \Aut{\Tree}_{\{\Omega_X\}}$ and any vertex $v \in V\Tree$ we define $
\sigma_{\mathcal L}(g,v) := (\lambda_{v}^{\mathcal L})^{-1} \, g \, \lambda_{v^g}^{\mathcal L}$. Our maps are applied from the right so for example if $v \in \Omega_X$ then $\sigma_{\mathcal L}(g,v): X \rightarrow X$ is a bijection. Thus for $v \in \Omega_X$ (resp.~$v \in \Omega_Y$) we have $\sigma_{\mathcal L}(g,v) \in \Sym{X}$ (resp.~$\sigma_{\mathcal L}(g,v) \in \Sym{Y}$). Thus we may define the following universal group,
\[
\cU_{\mathcal L}(M,N):= \left\{ g\in\Aut{\Tree}_{\{\Omega_X\}}\ {\big|}\ 
 \sigma_{\mathcal L}(g,v)\in M \ \forall v \in \Omega_X \ \text{ and } \ 
 \sigma_{\mathcal L}(g,v)\in N \ \forall v \in \Omega_Y
\right \}.
\]
By \cite[Proposition~11]{smith17}, the groups $\cU_{\mathcal L}(M,N)$ arising from different legal colourings are conjugate in $\Aut{\Tree}$. In the rest of the article we consider a fixed legal colouring $\mathcal L$ and we will write $\mathcal U(M,N)$ for $\cU_{\mathcal L}(M,N)$. The \emph{box product} $M\boxtimes N$ is the permutation group induced by $\cU(M,N)$ on $\Omega_Y$ (see \cite[Section~3]{smith17}) and the groups $\mathcal U(M,N)$ and $M\boxtimes N$ are isomorphic as topological groups.

As a topological group under the permutation topology, $\mathcal U(M,N)$ enjoys the following properties. The name \emph{universal locally-$(M, N)$ group} arises from item~\ref{item:thm:simon_duke:universal}.

\begin{thm}[{\cite[Theorem~1]{smith17}}]\label{thm:simon_duke} Suppose that $M\leq  \Sym{X}$ and $N \leq \Sym{Y}$ are permutation groups such that $\lvert X\rvert, \lvert Y\rvert >1$, with $M$ or $N$ nontrivial. Then the following hold:
\begin{enumerate}
    \item $\mathcal{U}(M, N) \leq \Aut{\Tree}$ is locally-$(M,N )$.
    \item \label{item:thm:simon_duke:universal} If $M$ and $N$ are transitive and $H \leq \Aut{\Tree}$ is locally-$(M,N )$, then $H$ is conjugate in $\Aut{\Tree}$ to some subgroup of $\mathcal{U}(M, N)$.
    \item If $M$ and $N$ are closed, then $\mathcal{U}(M,N)$ is closed.
    \item If M and N are generated by point stabilizers, then $\mathcal{U}(M, N)$ is simple if and only if $M$ or $N$ is transitive.
    \item If $M$ and $N$ are closed, then $\mathcal{U}(M, N)$ is locally compact if and only if all point stabilizers in $M$ and $N$ are compact.
    \item If $M$ and $N$ are closed and compactly generated with compact point stabilizers and only finitely many orbits, and $M$ or $N$ is transitive, then $\mathcal{U}(M, N)$ is compactly generated.
    \item $\mathcal{U}(M, N)$ is discrete if and only if $M$ and $N$ are semiregular.
\end{enumerate}
\end{thm}

Fix a vertex $v \in V\Tree$. We now describe the vertex stabilizer $\mathcal{U}(M, N)_v$ when $M$ and $N$ are transitive and closed. Our description is in terms of a vertex labelling and an arc labelling of $\Tree$ that depends on our choice of the vertex $v$.

Suppose $v \in \Omega_X$ and fix $x \in X$ and $y \in Y$. Notice that $\mathcal{U}(M, N)_v$ acts as automorphisms of the rooted tree $\Tree_v$ obtained from $\Tree$ by declaring $v$ to be the root. We will give $\Tree_v$ an arc colouring $\mathcal L'$ obtained from $\mathcal L$ in the following way. Let $b_v \in M$ be the identity. For each vertex $w \in V\Tree \setminus \{v\}$ let $a_w$ be the unique arc in $A(w)$ pointing towards the root $v$. Since $M$ and $N$ are transitive, if $w \in \Omega_X$ (resp.~$w \in \Omega_Y$) we can choose $b_w \in M$ (resp.~$b_w \in N$) such that the arc colouring $\mathcal L': A\Tree_v \rightarrow X \sqcup Y$ given by $\lambda_w^{\mathcal L} \, b_w$ on each $A(w)$ maps each $a_w$ to $x$ (if $w \in \Omega_X$) or $y$ (if $w \in \Omega_Y)$. 

While $\mathcal L'$ is not a legal arc colouring of $\Tree$, it is a natural arc colouring for a rooted tree, in which  all arcs pointing towards the root are coloured either $x$ or $y$ (depending on the level) and all arcs from any given vertex $w$ to its children are coloured bijectively with $X^o := X \setminus \{x\}$ or $Y^o := Y \setminus \{y\}$ (depending on the level of $w$). We use this to give an alternative description of $\mathcal{U}(M, N)_v$ that can be easily reconciled with an infinite iterated wreath product structure.

For $w \in V\Tree$ and $g \in \Aut{\Tree}_v$ recall $\sigma_{\mathcal L'}(g,w) = (\lambda_w^{\mathcal L'})^{-1} \, g \, \lambda_{w^g}^{\mathcal L'}$ where $\lambda_w^{\mathcal L'}$ is the restriction of $\mathcal L'$ to $A(w)$. A routine calculation and simple observations give the following:
\begin{enumerate}
\item
    $\sigma_{\mathcal L'}(g,w) = b_w^{-1} \sigma_{\mathcal L}(g,w) b_{w^g}$;
\item
    $a_{w^g} = a_w^g$; and
\item
    $a_w^{\mathcal L'}$ is $x$ if $w \in \Omega_X$ and $y$ if $w \in \Omega_Y$.
\end{enumerate}
We claim that the stabilizer $\mathcal{U}(M, N)_v$ is equal to
\[Q:= \set{g \in \Aut{T}_v}{\sigma_{\mathcal L'}(g,w)\in M_x \ \forall w \in \Omega_X \ \text{ and } \ 
 \sigma_{\mathcal L'}(g,w)\in N_y \ \forall w \in \Omega_Y}.\]
Indeed, if $g \in Q$ and $w \in \Omega_X$ then 
$b_w^{-1} \sigma_{\mathcal L}(g,w) b_{w^g} = \sigma_{\mathcal L'}(g,w)\in M$. Since $b_w \in M$ it follows that $\sigma_{\mathcal L}(g,w) \in M$. Similarly for $w \in \Omega_Y$ we have $\sigma_{\mathcal L}(g,w) \in N$. Since $g$ fixes $v$ it also preserves the parts of the bipartition and we thus have $g \in \mathcal{U}(M, N)_v$. On the other hand, if $g \in \mathcal{U}(M, N)_v$ and $w \in \Omega_X$ then by definition $\sigma_{\mathcal L}(g, w) \in M$ and therefore $\sigma_{\mathcal L'}(g,w) \in M$. Furthermore, $\sigma_{\mathcal L'}(g,w) \in M_x$ because
\[
x^{\sigma_{\mathcal L'}(g,w)} = x^{(\lambda_w^{\mathcal L'})^{-1} \, g \, \lambda_{w^g}^{\mathcal L'}} = a_w^{g \, \lambda_{w^g}^{\mathcal L'}} = a_{w^g}^{\lambda_{w^g}^{\mathcal L'}} = x.
\]
Similarly if $w \in \Omega_Y$ then $\sigma_{\mathcal L'}(g,w) \in N_y$. Our claim is thus established.

We can now use $\mathcal L'$ to identify vertices of $\Tree_v$ with the points in the following sets $\Omega_i$, which will allow us to determine the infinite iterated wreath product structure of $Q$ and thus $\mathcal{U}(M, N)_v$ following Section~\ref{section:WreathProds}.

Let $\Omega_1 := X$ and $\Omega_2 := Y^o \times X$ and $\Omega_3 := X^o \times Y^o \times X$, and so on. To simplify our notation for this alternating of sets, let $L_1 := X$ and for $i > 1$ define
\[
L_i :=
\begin{cases}
    Y^o \quad \text{if $i$ is even;}\\
    X^o \quad \text{if $i$ is odd.}
\end{cases}
\]
Then for $i>1$ we have $\Omega_{i} = L_i \times \Omega_{i-1}$.
As in Section~\ref{section:WreathProds} we construct a rooted tree $\Tree_\emptyword$ whose vertices are the points in the sets~$\Omega_{i}$ together with a root~$\emptyword$ in which the root~$\emptyword$ is joined to every vertex in~$\Omega_{1}$, while if $i \geq 2$ and $z = (\ell,\omega) \in \Omega_{i}$ where $\ell \in L_{i}$ and $\omega \in \Omega_{i-1}$ then there is an edge from~$\omega$ to~$z$. There is a natural identification between $\Tree_\emptyword$ and $\Tree_v$ that we now describe.
First we identify their roots. Then, for each vertex $w \in V\Tree_v \setminus \{v\}$ there is a directed geodesic path of arcs from $v$ to $w$, say $a_1, \ldots, a_i$ where the origin vertex of $a_1$ is $v$ and the terminal vertex of $a_i$ is $w$. We give $w$ the label $(a_i^{\mathcal L'}, \ldots, a_1^{\mathcal L'}) \in L_i \times \cdots \times L_1 = \Omega_i$ and identify vertices of $V\Tree_v$ with their label (which is a vertex in $\Tree_\emptyword$).
 In this way the iterated wreath product
\[
W = \dots \wr_{\Omega_{3}} M_x \wr_{\Omega_{2}} N_y \wr_{\Omega_{1}} M
\]
can be viewed as a subgroup of the stabilizer $\Aut{\Tree}_v$.

Suppose $w \in \Omega_X$. We may view $w$ according to its vertex label as an element in one of the sets $\Omega_{2i}$ and the children of $w$ as elements in $L_{2i+1} \times \Omega_{2i} = X^o \times \Omega_{2i}$. The action of $W$ on level $\Omega_{2i+1}$ of $\Tree_v$ is equal to that of $M_x \wr_{\Omega_{2i}} H$ where $H := N_y \wr_{\Omega_{2i-1}} \dots \wr_{\Omega_{3}} M_x \wr_{\Omega_{2}} N_y \wr_{\Omega_{1}} M$. Notice that for all vertices $w'$ in level $\Omega_{2i}$, for any arc $a \in A(w')$ to a child $(x_0, w') \in X^o \times \Omega_{2i}$ we have $(x_0, w')^{\mathcal L'} = x_0$. Write $M_x \wr_{\Omega_{2i}} H$ as $B \rtimes H$ where $B = M_x^{\Omega_{2i}}$.
Fix $g \in W \leq \Aut{\Tree}_v$ and let $bh \in B \rtimes H$ be the element of $M_x \wr_{\Omega_{2i}} H$ that has the same action as $g$ on $\Omega_{2i+1}$. Write $b = (m_{w'})_{w' \in \Omega_{2i}}$ as a sequence indexed by the set $\Omega_{2i}$. Then for all $x_0 \in X^o$ we have that $(x_0)^{\sigma_{\mathcal{L}'}(g,w)} = x_0^{m_w}$, and it then follows easily that $\sigma_{\mathcal{L}'}(g,w) \in M_x$. A similar argument shows that $\sigma_{\mathcal{L}'}(g,w) \in N_y$ whenever $w \in \Omega_Y$. Thus $W \leq Q$.

On the other hand, for any $g \in Q = \mathcal{U}(M, N)_v$ one can (using standard arguments, e.g.,~\cite[\S3.2]{BM00}) inductively build  elements $g_i \in M_x \wr_{\Omega_{2i}} \dots \wr_{\Omega_{3}} M_x \wr_{\Omega_{2}} N_y \wr_{\Omega_{1}} M$ such that $g$ and $g_i$ induce the same permutation on $\set{w \in V\Tree}{d(v,w) \leq 2i+1}$. The resulting sequence $(g_i)_{i \in \Nat}$ lies in $W \leq \mathcal{U}(M, N)_v$ and $\mathcal{U}(M, N)_v$ is closed in the permutation topology; thus $(g_i)_{i \in \Nat}$ in its action on $\Tree_v$ induces the same permutation of $V\Tree$ as $g$. Since $\mathcal{U}(M, N)_v$ is faithful on $V\Tree$, the two elements must coincide and hence $W = \mathcal{U}(M, N)_v$.

We have identified the structure of the vertex stabiliser $\mathcal{U}(M, N)_v$ when $v \in \Omega_X$. An entirely analogous argument can be used when $v \in \Omega_Y$ with $\Omega'_1 := Y$ and $\Omega'_2 := X^o \times Y$ and $\Omega'_3 := Y^o \times X^o \times Y$, and so on. In this case the stabiliser can be identified with the iterated wreath product $\dots \wr_{\Omega'_{3}} N_y \wr_{\Omega'_{2}} M_x \wr_{\Omega'_{1}} N$ which acts as above but on a rooted tree whose root lies in $\Omega_Y$. We summarise our remarks in the following proposition.

\begin{prop}\label{prop:point_stabs_box_prod}
Under the hypotheses of Theorem~\ref{thm:simon_duke}, if $M$ and $N$ are transitive and closed and $v \in V\Tree$ then, after making the identifications described above,
\[
\mathcal{U}(M, N)_v = 
    \begin{cases}
        \dots \wr_{\Omega_{3}} M_x \wr_{\Omega_{2}} N_y \wr_{\Omega_{1}} M &\text{if $v \in \Omega_X$}, \\
        \dots \wr_{\Omega'_{3}} N_y \wr_{\Omega'_{2}} M_x \wr_{\Omega'_{1}} N &\text{if $v \in \Omega_Y$.}
    \end{cases}
\]
\end{prop}

\section{Constructing a compactly generated simple tdlc group from an arbitrary linear profinite group}\label{sec:linear}

For this section, let $K$~be a non-Archimedean local field equipped with a discrete valuation~$\nu$ and let $\localring{K}$~be the ring of integers of~$K$.  As an additive group, $K$~is a tdlc group and $\localring{K}$~is profinite.  For our application in Section~\ref{sec:ProofUncountableSimple}, when we prove Theorem~\ref{thm:uncount_simple}, we shall take $K = \padics$.  We note, however, that other examples, such as $K = \Laurentseries{\Field{p}}{t}$, are also permitted in what we do here.

Now suppose that $H$~is a compact subgroup of~$\GL[n]{K}$ for some natural number~$n$.  We can then form the semidirect product $K^{n} \rtimes H$ using the natural action of~$H$.  We shall also consider $U \rtimes H$ where $U$~is an open $H$\nbd submodule of~$K^{n}$.

\begin{lem}
\label{lem:Semidirect-Trivialcore}
The group $H$ normalizes a compact open subgroup~$U$ of $K^{n}$.  Moreover the action of~$H$ on~$U$ is faithful, so the normal core of~$H$ in~$U \rtimes H$ is trivial.
\end{lem}

\begin{proof}
Let $W$~be any compact open subgroup of~$K^{n}$.  Using the continuity of the action, a standard compactness argument shows that $\set{ h \in H }{ W^{h} \leq W }$ is open in~$H$.  Using the compactness of~$H$, we deduce that $W$~has only finitely many $H$\nbd conjugates and hence $H$~normalizes some compact open subgroup~$U$ of~$K^{n}$.  If $h \in H$ centralises $U$ then it fixes all scalar multiples of elements of~$U$ and therefore the action of~$h$ on~$K^{n}$ is also trivial.  Since $H \leq \GL[n]{K}$ acts faithfully on $K^{n}$ it follows that $h = 1$.   Finally if $C$~denotes the core of~$H$ in~$U \rtimes H$, then $[C,U] \leq C \cap U \leq H \cap U = \1$, so $C$~is trivial.
\end{proof}

\begin{prop}
\label{prop:ExpansiveStep1}
Let $K$~be a non-Archimedean local field and $H$~be a compact subgroup of~$\GL[n]{K}$ acting on~$K^{n}$ in the natural way.  Let $\lambda \in K$ with valuation $\nu(\lambda) > 0$.  Let $\langle b \rangle$~be a discrete infinite cyclic group acting on~$K^{n}$ with the element~$b$ inducing multiplication by the scalar~$\lambda$.  Then
\[
E = K^{n} \rtimes (H \times \langle b \rangle)
\]
is a compactly generated tdlc group.  Furthermore, if $U$~is a compact open subgroup of~$K^{n}$ normalized by~$H$, then $UH$~has trivial core in $E$.
\end{prop}

\begin{proof}
If we equip the infinite cyclic group~$\langle b \rangle$ with the discrete topology, then we may define a continuous homomorphism $H \times \langle b \rangle \to \GL[n]{K}$ that maps each element of~$H$ to itself and $b$~to the scalar matrix~$\lambda I$.  We may therefore form the tdlc group $E = K^{n} \rtimes (H \times \langle b \rangle)$.  By Lemma~\ref{lem:Semidirect-Trivialcore}, there is a compact open subgroup~$U$ of~$K^{n}$ that is normalized by~$H$.  Since $b^{-r}Ub^{r} = \lambda^{r}U$ for all $r \in \Zint$, it follows that $K^{n}$~is contained in the subgroup generated by the $\langle b \rangle$\nbd conjugates of~$U$.  Therefore $E = \langle UH, b \rangle$ and consequently $E$~is indeed compactly generated.  Finally, the intersection of the $\langle b \rangle$\nbd conjugates of~$U$ is trivial and, hence, the intersection of the $\langle b \rangle$\nbd conjugates of~$UH$ equals~$H$.  Lemma~\ref{lem:Semidirect-Trivialcore} now implies that $UH$~has trivial core in~$E$.
\end{proof}

The following lemma involves a permutation group~$F$ on the set $\Omega = \{1,2,\dots,r\}$.  Given such~$F$, we define~$F^{+}$ to be the subgroup generated by its point stabilizers.  If $r = 1$ then $F = F^{+}$ are both trivial, while if $r = 2$ then $F^{+} = \1$ even when $F = \Sym{\Omega}$.  We assume $r \geq 3$ in the following to avoid these trivial cases.

\begin{lem}\label{lem:wreath_point_stabilizers}
Let $G$~be a compactly generated tdlc group and let $L$~be a compact subgroup with trivial core in $G$.  Let $r \geq 3$ and $F$~be a permutation group on $\Omega = \{1,2,\dots,r\}$ such that $F^{+}$~is transitive.  Form the wreath product $G \wr_{\Omega} F$ and define a homomorphism $\pi \colon G^{r} \to G/[G,G]L$ by
\[
(g_{1},g_{2},\dots,g_{r}) \mapsto [G,G]Lg_{1}g_{2}\dots g_{r}.
\]
Let $N = \ker\pi \rtimes F$.  Then
\begin{enumerate}
\item \label{i:N-GeneratedByConjugates}
$N$ is generated by the $N$-conjugates of $L \wr F$;
\item \label{i:N-FaithfulAction}
$N$ acts faithfully on the coset space of~$L \wr F$ in~$N$;
\item \label{i:N-CompactlyGen}
$N$ is compactly generated.
\end{enumerate}
\end{lem}

\begin{proof}
\ref{i:N-GeneratedByConjugates}~For distinct $i,j \in \Omega$ and $g \in G$, write~$\tau_{ij}(g)$ for the element of~$G^{r}$ whose $i$th coordinate equals~$g$, $j$th coordinate equals~$g^{-1}$ and all other entries are trivial.  Observe that $\ker\pi$~is the product of $[G,G]^{r}L^{r}$ with the subgroup generated by all~$\tau_{ij}(g)$ for distinct $i,j \in \Omega$ and $g \in G$.  Let $\sigma \in F$ be a permutation that fixes some $k \in \Omega$ and such that $i\sigma = j$.  Then we calculate that
\[
[\sigma,\tau_{ki}(g)] = \sigma^{-1} \tau_{ki}(g)^{-1} \sigma \tau_{ki}(g)
= \tau_{kj}(g^{-1}) \, \tau_{ki}(g) = \tau_{ji}(g).
\]
Hence the normal closure~$N^{\ast}$ of~$L \wr F$ in~$N$ contains~$\tau_{ij}(g)$ for $g \in G$ whenever such a permutation~$\sigma$ may be found in~$F$.

Now let $i,j \in \Omega$ be arbitrary.  There exists some permutation in~$F^{+}$ that moves~$i$ to~$j$.  Hence there is a sequence $i = i_{0}$, $i_{1}$, \dots, $i_{m} = j$ and permutations $\sigma_{1}$,~$\sigma_{2}$, \dots,~$\sigma_{m}$ each of which fix some point of~$\Omega$ and such that $\sigma_{k}$~moves~$i_{k-1}$ to~$i_{k}$.  The previous paragraph shows that $\tau_{i_{k-1}i_{k}}(g) \in N^{\ast}$ for each~$k$ and hence $N^{\ast}$~also contains
\[
\tau_{i_{0}i_{1}}(g) \, \tau_{i_{1}i_{2}}(g) \cdots \tau_{i_{m-1}i_{m}}(g) = \tau_{ij}(g).
\]
Thus $N^{\ast}$~contains~$\tau_{ij}(g)$ for every distinct pair $i,j \in \Omega$ and every~$g \in G$.  By definition, $L^{r} \leq N^{\ast}$, while if $i$,~$j$ and~$k$ are distinct points of~$\Omega$, then $[\tau_{ij}(g),\tau_{ik}(h)]$~is equal to the element of~$G^{r}$ with the commutator~$[g,h]$ in the $i$th coordinate and all other entries trivial.  We therefore deduce $[G,G]^{r} \leq N^{\ast}$ and hence $N^{\ast} = N$, as claimed.

\ref{i:N-FaithfulAction}~We shall show that the normal core~$C$ of~$L \wr F$ in~$N$ is trivial.  Let $v\sigma \in C$ where $v \in L^{r}$ and $\sigma \in F$.  Assume first that $\sigma \neq 1$.  Hence there exists some $i \in \Omega$ such that $i\sigma \neq i$.  Since $r \geq 3$, there exists $j \in \Omega$ such that $j \neq i$ and $j\sigma \neq i$.  Take $g \in G \setminus L$ and consider the commutator of~$v\sigma$ with the element~$\tau_{ij}(g)$, as defined in Part~\ref{i:N-GeneratedByConjugates}.  This is necessarily an element of~$C$ and hence contained in~$L \wr F$, but we calculate that
\[
[v\sigma, \tau_{ij}(g)] = (\tau_{ij}(g^{-1}))^{v\sigma} \, \tau_{ij}(g) = \tau_{i\sigma,j\sigma}(g^{-1})^{v^{\sigma}} \tau_{ij}(g)
\]
and its $i$th coordinate is equal to~$g \notin L$.  This contradiction establishes that $\sigma = 1$.  Therefore $C \leq L^{r}$ and hence $C$~is contained in intersection of all conjugates of~$L^{r}$ by the elements~$\tau_{ij}(g)$ for distinct $i,j \in \Omega$ and $g \in G$.  Since $L$~has trivial core in~$G$, we now conclude that the projection of~$C$ onto the $i$th coordinate of~$L^{r}$ is trivial and hence $C = \1$, as required.

\ref{i:N-CompactlyGen}~Let $X$~be a compact symmetric generating set for~$G$.  The map $g \mapsto \tau_{ij}(g)$ is continuous, so the set $Y = \set{ \tau_{ij}(x) }{ i,j \in \Omega, \; x \in X }$ is compact.  Observe that $\ker\pi$~is generated by $L^{r}$ together with~$Y$ and hence is compactly generated.  Since $\ker\pi$~has finite index in~$N$, it now follows that $N$~is also compactly generated.
\end{proof}

We now combine the two previous results and then list some properties of the groups constructed that will feed into Theorem~\ref{thm:simon_duke}.

Let $K$~be a non-Archimedean local field and $H$~be a compact subgroup of~$\GL[n]{K}$ with its natural action on~$K^{n}$.  Apply Lemma~\ref{lem:Semidirect-Trivialcore} to produce a compact open subgroup~$U$ of~$K^{n}$ that is normalized by~$H$. According to Proposition~\ref{prop:ExpansiveStep1} we can form the semidirect product $E = K^{n} \rtimes (H \times \langle a \rangle)$ where $a$~is an element of infinite order that acts on~$K^{n}$ by multiplying by a chosen scalar~$\lambda$ with positive valuation.  We have noted that this is a compactly generated tdlc group.  Set $L = UH$, which is a compact open subgroup of~$E$.  Now choose a permutation group~$F$ on some set $\Omega = \{1,2,\dots,r\}$, where $r \geq 3$, with the property that the subgroup $F^{+}$ generated by its point stabilizers is transitive.  Define $\pi \colon E^{r} \to E/[E,E]L$ by
\begin{equation}
(g_{1},g_{2},\dots,g_{r}) \mapsto [E,E]Lg_{1}g_{2}\dots g_{r}.
\label{eq:pi-map}
\end{equation}
Lemma~\ref{lem:wreath_point_stabilizers} tells us that $N = \ker\pi \rtimes F$ is a compactly generated tdlc group that is generated by the conjugates of $V = L \wr F$ in~$N$. Moreover, $V$ is a compact open subgroup of $N$.

One important comment is in order at this point: in the discussion below we will need topological properties of $N$ when we view it as a permutation group equipped with the permutation topology, rather than simply the topology arising from the construction described, in order to be able to apply Theorem~\ref{thm:simon_duke}.  We shall describe now how to view it as a permutation group and deduce its properties with the permutation topology from those established in Lemma~\ref{lem:wreath_point_stabilizers}.

Let $Y$~denote the coset space of~$V$ in~$N$. As remarked in Section~\ref{sec:topological_groups}, since $V$ is a compact open subgroup with trivial core in $N$,  we can identify $N$ with its Schlichting completion $\hat{N}\leq \Sym{Y}$ and, under this identification, we can view $N\leq \Sym{Y}$ as a transitive closed tdlc permutation group with compact open stabilizers (under the permutation topology). Moreover, the point stabilizers in this action of $N$ are the conjugates of $V$ and, by Lemma~\ref{lem:wreath_point_stabilizers}, $N$ is generated by point stabilizers. Additionally, since $N$ and $\hat{N}$ are topologically isomorphic, point \ref{i:N-CompactlyGen} of Lemma~\ref{lem:wreath_point_stabilizers} implies that $\hat{N}$ is compactly generated under the permutation topology. Finally, we observe that $N$ is not semiregular because $V$ is non-trivial.

We may therefore, with an appropriate choice of permutation group~$M$, use the above $N$ as input for the box product and universal group construction. Theorem~\ref{thm:simon_duke} and Proposition~\ref{prop:point_stabs_box_prod} then yield:

\begin{prop}
\label{prop:ApplicationOfBoxProduct}
Let $F$~be a permutation group of degree~$r \geq 3$ such that the subgroup~$F^{+}$ of~$F$ generated by point stabilizers is transitive, let $K$~be a non-Archimedean local field, $H$~be a compact subgroup of~$\GL[n]{K}$ and $U$~be a compact open $H$\nbd submodule of~$K^{n}$.  Write $N = \ker\pi \rtimes F$ where $\pi$~is the map given in Equation~\eqref{eq:pi-map} viewed as a permutation group on the coset space~$Y$ of $V = UH \wr F$ in~$N$.  Furthermore, let $M$~be a non-trivial closed transitive permutation group on some set~$X$ that is compactly generated, has compact point stabilizers, and is generated by its point stabilizers 
and let $S$~be the group $\mathcal{U}(M,N) \leq \Aut{\Tree}$.  Then
\begin{enumerate}
\item $S$~is a non-discrete, compactly generated and abstractly simple tdlc group;
\item \label{i:BoxAppStabilizers} For $v \in \Omega_X \subseteq V\Tree$ choose $x \in X$ and let $\Omega_1 := X, \, \Omega_{2i} := (Y\setminus \{V\})\times \Omega_{2i-1}$ and $\Omega_{2i+1} := (X\setminus \{x\})\times \Omega_{2i}$, then under the identification described in Section~\ref{sec:groups_acting_on_trees} we have that the stabilizer of $v$ in $S$ is an iterated wreath product:
\[S_v = 
\dots \wr_{\Omega_6} V \wr_{\Omega_5} M_{x} \wr_{\Omega_4} V \wr_{\Omega_3} M_{x} \wr_{\Omega_2} V \wr_{\Omega_1} M.
\]
\end{enumerate}
\end{prop}

\noindent
Recall that the point stabilizers are open in~$S$ since it is equipped with the permutation topology.  Hence Part~\ref{i:BoxAppStabilizers} describes certain open subgroups of~$S$.

\begin{cor}\label{cor:ApplicationOfBoxProduct}
Every profinite group which is linear over a non-Archimedean local field appears as a closed subgroup of a compactly generated simple tdlc group.
\end{cor}


\section{Distinguishing commensurability in some iterated wreath products} \label{sec:uniform}

\subsection{Snopce's groups and their properties}
\label{sub:Snopce}

Let $p$~be a prime number which for convenience we shall assume is odd.  (When we use the following in Section~\ref{sec:ProofUncountableSimple} we shall further specialize to $p \geq 7$, but that is not necessary for the content of this section.)  For $\alpha \in \units{\padicZ}$, we construct the
$\padicZ$\nbd Lie ring~$L_{3}(\alpha)$ which, as a $\padicZ$\nbd
module, is free with basis~$\{x,e_{2},e_{3}\}$ and where the Lie
bracket is given by
\[
  [e_{2},e_{3}] = 0, \qquad [e_{2},x] = \alpha e_{3}, \qquad
  [e_{3},x] = e_{2} + e_{3}.
\]
As noted in~\cite[Proposition~3.1]{snopce}, if $\alpha,\beta \in
\units{\padicZ}$, then $L_{3}(\alpha) \cong L_{3}(\beta)$ if and only
if $\alpha = \beta$.  Furthermore, $\alpha$~is an invariant of the
isomorphism type of the $\padics$\nbd Lie algebra $\padics
\otimes_{\padicZ} L_{3}(\alpha)$.  Since $p$~is odd, $pL_{3}(\alpha)$~is a powerful $\padicZ$\nbd Lie ring.

Let $G_{3}(\alpha)$~be the uniform pro\nbd$p$ group corresponding
to~$pL_{3}(\alpha)$.  Thus $G_{3}(\alpha)$~is generated, as a
pro\nbd$p$ group, by three elements $y$,~$z_{2}$ and~$z_{3}$ such that
$\overline{ \langle z_{2},z_{3} \rangle} \cong \padicZ \oplus \padicZ$
and $G_{3}(\alpha) \cong (\padicZ \oplus \padicZ) \rtimes \padicZ$.
The action of~$y$ on the normal subgroup is determined by the
$\padicZ$\nbd Lie ring~$pL_{3}(\alpha)$, but we shall not refer to
an explicit presentation for~$G_{3}(\alpha)$ in our argument.  Nevertheless, we remark that the group~$G_3(\alpha)$ fits the blueprint of Section~\ref{sec:linear} (see, for example, Proposition~\ref{prop:ExpansiveStep1}) when we take a subgroup~$H$ of~$\GL[2]{\padics}$ isomorphic to~$\padicZ$ acting as prescribed by the $\padicZ$\nbd Lie ring~$pL_{3}(\alpha)$ with the compact $H$\nbd submodule $U = \padicZ^{2}$ of $\padics^{2}$.  The group $G_3(\alpha)$ is $U\rtimes H$ with respect to this action.

We begin with an observation that will inform our choice of the parameter~$\alpha$.
In the following argument, when $X$~is a non-empty subset of~$\padics$, write~$(X)^{2}$ for the set of squares $\set{ x^{2} }{ x \in X }$.

\begin{lem}\label{lem:alpha_uncount}
  The set $\mathcal{A} = \set{ \alpha \in \padicZ }{ 1 + 4\alpha
    \notin (\padics)^{2} }$ is uncountable.
\end{lem}

\begin{proof}
  Since $\units{\padics} = \langle p \rangle \times \units{\padicZ}$,
  \[
  (\units{\padics})^{2} = \langle p^{2} \rangle \times
  (\units{\padicZ})^{2}
  \]
  and hence
  \[
  \padicZ \cap (\padics)^{2} = \{0\} \cup \set{ p^{2k}u^{2} }{ k \geq
    0, \; u \in \units{\padicZ} }.
  \]
  Therefore if $v \in \padicZ$ such that its image in $\padicZ /
  p\padicZ \cong \Field{p}$ is not a square, then $v \notin
  (\padics)^{2}$.  It follows that $\padicZ \setminus (\padics)^{2}$,
  and hence $\mathcal{A} = \set{ \tfrac{1}{4}(\beta-1) }{ \beta \in
    \padicZ \setminus (\padics)^{2}}$, is uncountable.
\end{proof}

Choose $\alpha \in \mathcal{A}$.  Define
\[
A = \begin{pmatrix}
  0 & \alpha \\ 1 & 1
\end{pmatrix},
\]
which is the matrix of the restriction of~$\ad{x}$ to the space
$\Span[\padicZ]{e_{2},e_{3}}$.  The characteristic polynomial of~$A$
is
\[
c_{A}(X) = \det \begin{pmatrix}
  X & -\alpha \\ -1 & X-1
\end{pmatrix}
= X^{2} - X - \alpha,
\]
so the eigenvalues of~$A$ are
\[
\half \bigl( 1 \pm \sqrt{ 1 + 4\alpha } \bigr).
\]
Since $\alpha \in \mathcal{A}$, these eigenvalues do not lie
in~$\padics$.  Furthermore, if $r \geq 1$, the matrix of the
restriction~$\ad{(p^{r}x)}$ to the space
$\Span[\padicZ]{p^{r}e_{2},p^{r}e_{3}}$ is~$p^{r}A$, whose eigenvalues
are $\half p^{r} \bigl( 1 \pm \sqrt{ 1 + 4\alpha } \bigr)$, none of
which lie in~$\padics$.

\begin{lem}
  \label{lem:LieIdeal}
  Let $\alpha \in \mathcal{A}$, let $s \geq 1$ and let $L = p^{s}
  L_{3}(\alpha)$.  If $I$~is an ideal of~$L$ contained
  in~$\Span[\padicZ]{p^{s}e_{2},p^{s}e_{3}}$, then the rank of~$I$ is
  either~$0$ or~$2$.
\end{lem}

\begin{proof}
  Suppose that $\beta,\gamma \in \padicZ$ such that $I =
  \Span[\padicZ]{\beta p^{s}e_{2} + \gamma p^{s}e_{3}}$ is an ideal
  of~$L$.  Then $[\beta p^{s}e_{2} + \gamma p^{s}e_{3}, p^{s}x] =
  \lambda (\beta p^{s}e_{2} + \gamma p^{s}e_{3})$ for some $\lambda
  \in \padicZ$.  Hence $(\beta,\gamma)$~is an eigenvector of~$p^{s}A$
  with eigenvalue~$\lambda$, which is a contradiction.
\end{proof}

\begin{cor}
  \label{cor:Homs}
  Let $\alpha \in \mathcal{A}$, let $r \geq 0$ and $\phi$~be a
  surjective homomorphism from the closed
  subgroup~$G_{3}(\alpha)^{p^{r}}$ of~$G_{3}(\alpha)$ onto a uniform
  pro\nbd$p$ group~$H$.  Then, up to isomorphism, either $H = \1$,
  $\padicZ$ or~$G_{3}(\alpha)^{p^{r}}$.
\end{cor}

\begin{proof}
  Let $L_{H}$~be the $\padicZ$\nbd Lie ring associated to~$H$.
  Note that $L = p^{r+1}L_{3}(\alpha)$~is the $\padicZ$\nbd Lie ring
  associated to~$G_{3}(\alpha)^{p^{r}}$.  According to
  \cite[Theorem~9.10]{DDMS}, there is a corresponding surjective
  $\padicZ$\nbd Lie ring homomorphism $\phi_{\ast} \colon L \to
  L_{H}$.  Write $M = \Span[\padicZ]{p^{r+1}e_{2},p^{r+1}e_{3}}$.  By
  Lemma~\ref{lem:LieIdeal} applied with $s = r+1$, $\ker\phi_{\ast} \cap M$~is either $\0$ or~$M$.  If $\ker\phi_{\ast} \cap M = \0$ then it must be the case that $\ker\phi_{\ast} = \0$ and therefore $\phi$~is an isomorphism.
  Otherwise $M \leq \ker\phi_{\ast}$ and hence either $\ker\phi_{\ast}
  = M$ or $\ker\phi_{\ast} = L$.  We deduce that in these cases either
  $H \cong \padicZ$ or $H = \1$.
\end{proof}

\subsection{Non-commensurability of certain iterated wreath products}
\label{sub:IteratedWreathLocal}

In this subsection we will consider certain iterated wreath products $W(\alpha)$ involving the groups $G_3(\alpha)$, and their Sylow pro-$p$ subgroups. Moreover, we will prove that $W(\alpha)$ and $W(\beta)$ cannot be commensurable for $\alpha\neq \beta\in \mathcal{A}$. The vertex stabilizers in the groups $S$ from Proposition~\ref{prop:ApplicationOfBoxProduct} are a particular case of the wreath products that we consider here. 

Consider $\alpha \in \mathcal{A}$ and fix a pro\nbd$p'$ group~$F$ acting continuously on some set~$\Omega$.
Define $V(\alpha) = G_{3}(\alpha) \wr_{\Omega} F$. Note that, by choice of~$F$ as a
$p'$\nbd group, $G_{3}(\alpha)^{\Omega}$~is the Sylow pro\nbd$p$ group
of~$V(\alpha)$.  We then construct an iterated wreath product in the
following way.  

Suppose that $(M_{n})_{n \in \Nat}$ is a sequence of pro\nbd$p'$ groups. Let $J_{1} = M_{1}$ and pick a continuous action of $M_{1}$ on a discrete set $X_1$.
Define $J_{2} = V(\alpha) \wr_{X_{1}} M_1$, equipped with the product topology so that $J_{2}$~is a profinite group.  Suppose, as an inductive hypothesis, that we have constructed a sequence $J_{1}$,~$J_{2}$, \dots,~$J_{2n}$ of iterated wreath products.  Pick a continuous action of~$J_{2n}$ on an infinite discrete set~$X_{2n}$ and define $J_{2n+1} = M_{n+1} \wr_{X_{2n}} J_{2n}$.  Again we equip~$J_{2n+1}$ with the product topology.  Now pick a continuous action of~$J_{2n+1}$ on an infinite discrete set~$X_{2n+1}$ and define $J_{2n+2} = V(\alpha) \wr_{X_{2n+1}} J_{2n+1}$, again viewed as a profinite group.  In this way, we construct an inverse system
\[
\dots \to J_{3} \to J_{2} \to J_{1} \to \1
\]
of profinite groups.  Finally set $W = W(\alpha) = \invlim J_{n}$.

We now define an associated sequence of pro\nbd$p$ groups that will coincide with a Sylow pro-$p$ subgroup of $W(\alpha)$.  Define $P_{1} = G_{3}(\alpha)^{\Omega \times X_{1}} = \bigl(G_{3}(\alpha)^{\Omega}\bigr) \wr_{X_{1}} \1$, which is a closed subgroup of~$J_{2}$ and is the Sylow pro\nbd$p$ group of~$J_{2}$.
Suppose that we have constructed a Sylow pro\nbd$p$ subgroup~$P_{n}$ of~$J_{2n}$.  Define $P_{n+1} = \bigl(G_{3}(\alpha)^{\Omega}\bigr) \wr_{X_{2n+1}} P_{n}$.  Since $M_{n+1}$ is a $p'$\nbd group, $P_{n+1}$~is a Sylow pro\nbd$p$ subgroup of $J_{2n+2} = V(\alpha) \wr_{X_{2n+1}} (M_{n+1} \wr_{X_{2n}} J_{2n} )$.  Set $P = \invlim P_{n}$, which we view as a subgroup of~$W$ in the natural way.  Then $P$~is a Sylow pro\nbd$p$ subgroup of~$W$ and, by construction, each~$P_{n}$ is a
torsion-free pro\nbd$p$ group.  We shall write $\pi_{n} \colon P \to P_{n}$ for the natural maps determined by the inverse limit.

\begin{lem}
  \label{lem:Main}
  Let $\alpha,\beta \in \mathcal{A}$.  Suppose that $W = W(\alpha)$
  has a closed subgroup that is isomorphic to $G_{3}(\beta)$ for
  some~$\beta$.  Then $\alpha = \beta$.
\end{lem}

\begin{proof}
  Since $P$~is a Sylow pro\nbd$p$ subgroup of~$W$, the hypothesis
  ensures that $P$~has a closed subgroup~$K$ that is isomorphic
  to~$G_{3}(\beta)$.  In particular, $K$~is non-abelian, so there
  exists some minimal~$n$ such that $K\pi_{n}$~is a non-abelian closed
  subgroup of~$P_{n}$.

  Let $\set{ \Gamma_{i} }{ i \in I }$ be the orbits of~$P_{n-1}$
  on~$X_{2n-1}$.  Thus the base group of~$P_{n}$ is the Cartesian
  product of factors $G_{3}(\alpha)^{\{\omega\} \times \Gamma_{i}}$
  for $\omega \in \Omega$ and $i \in I$, each of which is a normal
  subgroup of~$P_{n}$.  There is therefore, for each $\omega \in
  \Omega$ and $i \in I$, a natural homomorphism $\psi_{\omega,i}
  \colon P_{n} \to G_{3}(\alpha) \wr_{\Gamma_{i}} P_{n-1}$
  corresponding to projecting the base group of~$P_{n}$ onto
  $G_{3}(\alpha)^{\{\omega\} \times \Gamma_{i}}$.  Since the
  intersection of all the kernels of the maps~$\psi_{\omega,i}$ is
  trivial, there is an associated embedding of~$P_{n}$ in the
  Cartesian product
  \[
  \prod_{(\omega,i) \in \Omega \times I} (G_{3}(\alpha)
  \wr_{\Gamma_{i}} P_{n-1}).
  \]
  As $K\pi_{n}$~is non-abelian, it follows that there exists some
  $\omega \in \Omega$ and $i \in I$ such that
  $K\pi_{n}\psi_{\omega,i}$~is non-abelian.  Fix these indices and
  write $\Gamma = \Gamma_{i}$ and $\psi = \pi_{n}\psi_{\omega,i}$ to
  simplify notation.  Thus $\psi \colon K \to G_{3}(\alpha)
  \wr_{\Gamma} P_{n-1}$ is a homomorphism with non-abelian image.
  Note that $K\psi$~is a uniform pro\nbd$p$ group since it is an image
  of~$K$ in a torsion-free group.

  Let $C$~be the kernel of the action of~$P_{n-1}$ on~$\Gamma$.  Since
  $\Gamma$~is finite, $C$~has finite index in~$P_{n-1}$.  Hence there
  exists~$r$ such that
  \begin{equation}
    (K^{p^{r}})\psi = (K\psi)^{p^{r}} \leq G_{3}(\alpha) \wr_{\Gamma}
    C \cong G_{3}(\alpha)^{\Gamma} \times C.
    \label{eq:J-power}
  \end{equation}
  Note $(K^{p^{r}})\psi$~is non-abelian since $K\pi_{n}$~is uniform
  and non-abelian.  Furthermore, $(K^{p^{r}})\psi$~is the subdirect
  product of its images in each factor~$G_{3}(\alpha)$ and in~$C$
  arising from the above direct product.  The image in~$C$ is abelian
  since the image of~$K$ in~$P_{n-1}$ is, by choice of~$n$, abelian.
  Hence the image of~$(K^{p^{r}})\psi$ in one of the
  factors~$G_{3}(\alpha)$ in Equation~\eqref{eq:J-power} must be
  non-abelian.  Thus there is a homomorphism $\theta \colon K^{p^{r}}
  \to G_{3}(\alpha)$ with non-abelian image that is determined by
  projecting~$(K^{p^{r}})\psi$ to this factor.

  We now apply Corollary~\ref{cor:Homs} to conclude that $\theta$~is
  an isomorphism from~$K^{p^{r}}$ to its image in~$G_{3}(\alpha)$.  In
  particular, $(K^{p^{r}})\theta$~has rank~$3$ and hence is open
  in~$G_{3}(\alpha)$.  We conclude that there is an open subgroup
  of~$G_{3}(\beta)$ that is isomorphic to an open subgroup
  of~$G_{3}(\alpha)$; that is, $G_{3}(\alpha)$~and~$G_{3}(\beta)$ are
  commensurable.  Hence $\alpha = \beta$ by the proof of
  \cite[Theorem~1.1]{snopce}.
\end{proof}

\begin{cor}
\label{cor:W-nonCommensurable}
  The wreath products~$W(\alpha)$ for $\alpha \in \mathcal{A}$ are
  pairwise non-commensurable.
\end{cor}

\begin{proof}
  Let $\alpha,\beta \in \mathcal{A}$ and suppose that
  $W(\alpha)$~and~$W(\beta)$ are commensurable.  Then there exist open
  subgroups $U$~and~$V$ of $W(\alpha)$~and~$W(\beta)$, respectively,
  that are isomorphic.  Then there exists~$n$ such that $V$~contains
  the kernel of the map $W(\beta) \to W_{n}(\beta)$.  Therefore
  $V$~contains (many!) closed subgroups isomorphic to~$G_{3}(\beta)$.  Hence
  $U$, and therefore also~$W(\alpha)$, also has a closed subgroup isomorphic
  to~$G_{3}(\beta)$.  Lemma~\ref{lem:Main} then tells us that $\alpha
  = \beta$.
\end{proof}

\section{Proof of Theorem~\ref{thm:uncount_simple}}
\label{sec:ProofUncountableSimple}

We now describe the choices involved in constructing uncountably many pairwise not locally isomorphic groups in $\mathscr{S}$.

Let $\mathcal{A}$~be the subset of~$\padicZ$ defined in Lemma~\ref{lem:alpha_uncount}.
Fix $\alpha \in \mathcal{A}$.  Take a prime number $p \geq 7$ and choose a subgroup~$H$ of~$\GL[2]{\padics}$, isomorphic to the additive group~$\padicZ$ of $p$\nbd adic integers, such that the semidirect product $\padicZ^{2} \rtimes H$, constructed via the natural action of~$H$, is isomorphic to the uniform pro\nbd$p$ group~$G_{3}(\alpha)$ defined in Subsection~\ref{sub:Snopce}.  The observations made in Section~\ref{sec:linear} then apply with this choice of compact subgroup~$H$.

Let $r \geq 3$ and fix a permutation group~$F$ on $\Omega = \{1,2,\dots,r\}$ such that the subgroup~$F^{+}$ of~$F$ generated by point stabilizers is transitive and form the semidirect product $E = E(\alpha) = \padics^{2} \times (H \times \langle a \rangle)$ where $a$~acts on~$\padics^{2}$ by multiplying by~$p$.  We take $L = G_{3}(\alpha)$ and let $N = N(\alpha) = \ker\pi \rtimes F$, where $\pi$~is the map given in Equation~\eqref{eq:pi-map}.

Finally we take $M$~to be the alternating group~$A_{6}$ in its natural action on the set $X=\{1,\ldots,6\}$.  Applying Proposition~\ref{prop:ApplicationOfBoxProduct}, we establish that the universal group $S(\alpha) = \mathcal{U}(A_{6}, N(\alpha))$ is a closed permutation group on the $(6,\aleph_0)$-biregular tree $\Tree$ and it is a non-discrete, compactly generated and abstractly simple tdlc group. Furthermore, recalling that we denote by $\Omega_X$ the set of vertices of $\Tree$ with valency~$6$, in Proposition~\ref{prop:ApplicationOfBoxProduct} part~\ref{i:BoxAppStabilizers} we saw that the point stabilizer of a vertex $w\in \Omega_X$ in $S(\alpha)$ is of the form
\[
W(\alpha) = \dots \wr V(\alpha) \wr A_{5} \wr V(\alpha) \wr A_{5} \wr V(\alpha) \wr A_{6}.
\]
Our choice of~$p$ ensures that both $A_{5}$~and~$A_{6}$ are $p'$\nbd groups and hence this iterated wreath product~$W(\alpha)$ has the form considered in Subsection~\ref{sub:IteratedWreathLocal}.  Corollary~\ref{cor:W-nonCommensurable} then tells us that if $\alpha,\beta \in \mathcal{A}$ are distinct then the open subgroup~$W(\alpha)$ of~$S(\alpha)$ is not commensurable to the open subgroup~$W(\beta)$ of~$S(\beta)$.  Hence $S(\alpha)$~is not locally isomorphic to~$S(\beta)$ whenever $\alpha,\beta \in \mathcal{A}$ are distinct.  This establishes that $\set{S(\alpha)}{\alpha \in \mathcal{A}}$ is our claimed collection of $2^{\aleph_{0}}$~many non-discrete compactly generated simple tdlc groups which are pairwise not locally isomorphic.


\printbibliography
\end{document}